\documentclass{article}
\usepackage[utf8]{inputenc}
\usepackage{mathrsfs}
\usepackage{cite}
\usepackage{amsmath}
\usepackage[all,cmtip,2cell]{xy}
\usepackage{mathrsfs}
\usepackage{hyperref}
\usepackage{lineno}
\usepackage{mathrsfs}
\usepackage{amsfonts}
\usepackage{amsthm}
\usepackage{amssymb}
\usepackage{breakcites}
\usepackage[letterpaper, margin=1in]{geometry}
\numberwithin{equation}{section}
\theoremstyle{definition}
\newtheorem{define}{Definition}[section]
\newtheorem{example}[define]{Example}

\theoremstyle{remark}
\newtheorem{remark}[define]{Remark}
\theoremstyle{plain}
\newtheorem{theo}[define]{Theorem}
\newtheorem{lemma}[define]{Lemma}
\newtheorem{prop}[define]{Proposition}
\newtheorem{cor}[define]{Corollary}

\newcommand{\C}{\mathscr C}
\newcommand{\D}{\mathscr D}
\newcommand{\F}{\mathscr F}
\newcommand{\E}{\mathscr E}

\newcommand{\basetopos}{\mathscr S}
\newcommand{\bb}{\mathbb}
\newcommand{\mfrk}{\mathfrak}
\newcommand{\set}{\mathbf{Set}}
\newcommand{\cat}{\mathbf{Cat}}
\newcommand{\sub}{\mathrm{Sub}}

\title{A Topos View of Blockchain Consensus Protocols}
\author{Michael Lambert}
\date{November 2021}

\begin{document}

\maketitle

\begin{abstract}
    This paper presents a reformulation in topos logic of a safety result arising in an abstract presentation of blockchain consensus protocols. That is, in a high-level template for ``correct-by-construction" consensus protocols, it is shown that a proposition and its negation cannot both be safe in protocol states that have executions to some common state. This is in fact true for any inconsistent propositions and the proof requires only intuitionistic reasoning. This opens the door for work on consensus protocols in the internal language of a topos. As a first pass on such a program, the main contribution of this paper is the formulation of estimate safety in abstract correct-by-construction protocols as a forcing statement in the internal logic of a given topos. This is illustrated first in the setting of copresheaf toposes. It is also seen there that safety can be viewed as a modal statement. For these interpretations, some extensions and adaptations of results in the literature on modal operators in toposes are presented. The final reformulation of estimate safety is a completely elementary version in the language of an arbitrary topos where it is seen that estimate safety is equivalent to a certain forcing statement.
\end{abstract}

\tableofcontents

\section{Introduction}

A \textbf{distributed system} is a network of interconnected nodes tasked with solving certain computational problems. A long-standing issue is that of desiging algorithms or protocols by which distributed systems arrive at a \textbf{consensus} about a state of affairs, given network latency or that some nodes may be faulty or inoperative. These faults could be due to mechanical errors or the presence of malicious actors.

A \textbf{blockchain} is a distributed system in which nodes validate blocks forming a public ledger of transactions. Any blockchain is ultimately a linear order of individual blocks, each of which identifies its unique immediate predecessor, contains an identification hash, and records network transaction data. The blockchain \textbf{protocol} governs how nodes communicate. As part of the protocol, a consensus algorithm for such a blockchain is a process by which nodes come to agree on which blocks to include in the chain in the case that blocks are minted simultaneously or contain conflicting data. For example, \cite{Nakamoto} outlines the consensus algorithm used in the Bitcoin blockchain; \cite{EthWhitePaper} is the Ethereum blockchain whitepaper from 2013. \textbf{Safety} is the issue of whether nodes will or will not validate conflicting blocks. A proposition about a configuration of the system or the consensus is \textbf{safe} in a protocol state if it is validated by all subsequent states accessible from the given state. A consensus algorithm should guarantee safety given that a certain percentage of nodes are not faulty. See \cite{Gramoli} for a recent overview of issues in consensus protocols and fault tolerance in blockchain development.

Here the concern is the mathematics of an abstract version of the consensus protocol used in the Ethereum blockchain. Ethereum follows a ``correct-by-construction'' (CBC) Byzantine fault tolerant (BFT) protocol outlined for example in \cite{ButerinGriffith}, \cite{Buterin}, \cite{VladCasper} and \cite{ZamAndFam}. An abstract template (referred to here as ``Abstract CBC'') for such protocols was given \cite{AbstractCBC}. The present purpose is not to extend this ongoing work, but rather to explicitly formulate Abstract CBC in topos logic and show that the central concept of ``estimate safety" is equivalently a forcing statement and indeed a modal statement arising from certain geometric models as in \cite{Steve}. To this end, review the main definition of the Abstract CBC template as presented in \cite{AbstractCBC}.

\begin{define} An \textbf{estimate consensus protocol} consists of
\begin{enumerate}
    \item a set $C$ of possible \textbf{consensus values};
    \item a propositional logic $\mathcal L_C$ such that each proposition is either true or not true of each consensus value;
    \item a category $\Sigma$ whose objects $w$ are \textbf{protocol states} and whose morphisms are \textbf{protocol executions};
    \item finally an \textbf{estimator function} $E\colon \Sigma_0 \to \mathcal L_C$ assigning a proposition to each object of $\Sigma$ in such a way that if $Ew \Rightarrow p$, then $\neg (Ew \Rightarrow \neg p)$ for any proposition $p$ and any state $w$.
\end{enumerate}
Display an estimate consensus protocol as a tuple $(C, \mathcal L_C, \Sigma, E)$.
\end{define}

\begin{define}
    A proposition $p$ \textbf{safe in the protocol state} $w$ if for any execution $w \to w'$, it follows that $Ew'\Rightarrow p$ holds. Write `$S(p, w)$' to indicate that $p$ is safe in $w$. States $w_1$ and $w_2$ are \textbf{compatible} if there is a third state $w_3$ and executions $w_1\to w_3 \leftarrow w_2$. Say that $w_3$ is a \textbf{common future state}.
\end{define}

The main result of \cite{AbstractCBC} is then the following.

\begin{theo}
    A proposition $p$ and its negation $\neg p$ are not both safe in compatible states. In notation, for any compatible $w_1$ and $w_2$ and any proposition $p$, it is not the case that both $S(p,w_1)$ and $S(\neg p, w_2)$.
\end{theo}

This suggests a modest generalization, proved below as Theorem \ref{THEOREM1}.

\begin{theo}
    Inconsistent propositions are not both safe in compatible states. In notation, if $p\wedge q = \bot$ holds, then for any compatible states $w_1$ and $w_2$ it is not the case that both $S(p,w_1)$ and $S(q, w_2)$. 
\end{theo}

This generalization is proved using an intuitionistic logic $\mathcal L$ and by reasoning without double negation or DeMorgan's laws in the metatheory. From these intuitionistic proofs, the possibility arises of giving a topos-theoretic reformulation of consensus protocols. And indeed the main observations of this paper are that estimate safety is (1) equivalent to a forcing statement and (2) equivalent to a modal statement with a forcing semantics. As background, recall (e.g. \cite{Kripke}, \cite{KripkeModal}) that the traditional forcing semantics of a modal operator $\Box$ for ``necessity'' are $w\Vdash \Box p$ if, and only if, $v\Vdash p$ for all $w\leq v$. That is, ``necessarily $p$" is ``forced" at stage $w$ if it is also at all accessible future states. Now, ``safety in a state'' is inherently modal. Think of subsequent protocol states as \textbf{accessible future states}. Thinking of $C$ as possible configurations of the blockchain itself, the estimator $E$ returns a ``fork choice'' at a given protocol state, that is, a choice of which blocks to include and which to drop from the ledger. In this way, $p$ is safe at the protocol state if, and only if, $p$ is in the fork choice of every subsequent protocol state.

The question is how to make sense of these heuristics. A first approach is to use an adaptation of the relativised forcing relation `$\Vdash_*$' of \cite{Steve} induced from the canonical geometric morphism between presheaf toposes associated to the estimator $E\colon \Sigma\to PC$. This requires some technical results related to `$\Vdash_*$', such as Lemma \ref{lemma:SemanticsofTotalCosieve}, which are made as needed in \S \ref{section:ForcingSemanticsofEstimateSafety}. This leads to the first main result, Theorem \ref{theo:PresheafForcingSemanticsofSafety}, showing that estimate safety is equivalent to a $\Vdash_*$-forcing statement. In \cite{Steve} the presheaf semantics of modalized statements are given in terms of this relative forcing relation. Thus, along the way in \S \ref{section:GeometricModels}, some observations are made concerning $\Box$-semantics arising from ``geometric models'' as preparation for the second main result. Namely, Theorem \ref{cor:PresheafSafetyisModal} shows that estimate safety is also equivalent to a certain modal statement with semantics given by relative forcing.

Both of these preliminary results suffer from a defect. That is, for the semantics to work properly, it needs to be assumed that $E\colon \Sigma\to PC$ induces a surjective geometric morphism of copresheaf toposes, hence a geometric model. Since $PC$ is a poset category, this would require that $E$ is surjective on objects. This is probably an unreasonable assumption. For in \cite{VladCasper} and \cite{ZamAndFam}, the protocol $\Sigma$ is inductively \emph{derived from} transactions on the blockchain. This is a latency issue that is fundamental to blockchain development. So, not only would a completely elementary presentation of Abstract CBC protocols be preferred, but the use of geometric models, while it is intuitive and rather starkly illustrates the forcing and modal semantics, should ultimately be avoided in the settled account. The presheaf are left in because they explicitly illustrate the role of the ``total cosieve'' in Lemma \ref{lemma:SemanticsofTotalCosieve}. This is the basis of the elementary axiomatization and main forcing result. That is, \S \ref{section:ElementarySafetyInaTopos} gives a formulation of Abstract CBC protocols in the internal category theory of a topos. Safety is shown to be a forcing statement in Theorem \ref{theorem:ElementaryForcingSemanticsOfSafety} and the main safety result is proved in Theorem \ref{theoToposSafety}.

Results of \cite{Steve} are extended and adapted where needed. In particular, there are the following technical contributions to the theory of forcing semantic of modalized topos logic:
\begin{enumerate}
    \item Theorem \ref{theorem:BoxSemanticsElementaryVersion} extends the forcing semantics of `$\Box$' in \cite{Steve} from inclusions $|\C|\to \C$ to functors $\C\to\D$ that induce a surjection between copresheaf toposes.
    \item Lemma \ref{lemma:InternalTechnicalResult} shows that one of the main properties of forcing semantics can be internalized to general diagram categories in an arbitrary base topos.
    \item This allows a rewriting of Theorem \ref{theorem:BoxSemanticsElementaryVersion} to give forcing semantics of a modal operator induced by a geometric morphism of internal diagram categories in any topos. This appears as Theorem \ref{theorem:InternalForcingSemantics}.
\end{enumerate}
The following section \S \ref{section:HeytingAlgebras} gives an overview of Heyting algebras and some intuitionistic principles. Following that \S \ref{section:HeytingValuedSafety} gives an account of safety in a state is given using an intuitionistic system $\mathcal L$ in the place of the classical propositional system. Topos logic occupies the following section \S \ref{section:ToposesInternalLogic}. The final three sections give the final topos-theoretic account of Abstract CBC and the safety result, namely, Theorem \ref{theoToposSafety}. This paper is essentially an application of topos theory, but many details have been included in an effort to make it accessible for non-specialist readers interested in the mathematics of blockchain consensus protocols.

\textbf{Acknowledgements.} This paper was prepared under the supervision of Dr. Geoff Cruttwell at Mount Allison University where the author is a postdoctoral research associate. The author would like to thank Dr. Cruttwell for his support and feedback on an earlier version of the paper.

\section{Heyting Algebras}
\label{section:HeytingAlgebras}

Denote the top element of any lattice by `$\top$' and the bottom by `$\bot$'. Conjunction and disjunction are denoted by `$\wedge$' and `$\vee$', respectively. Most of the following is standard from \S I.8 of \cite{MM} or Chapter 1 of \cite{Handbook}. These details are included since the subobject classifier $\Omega$ in a topos is an internal Heyting algebra and thus obeys arrow-theoretic versions of the logical laws presented here. Throughout use `$\equiv$' as a metasymbol for ``if, and only if.''

\begin{define} 
    A \textbf{Heyting algebra} is a finitely complete lattice $\mathcal H$ where for each $p\in\mathcal H$, the functor $p\wedge -\colon \mathcal H \to\mathcal H$ has a right adjoint $p\Rightarrow (-)\colon \mathcal H\to\mathcal H$.
\end{define}

By the definition of adjoint functors, the Heyting implication satisfies
\begin{equation} \label{HeytingImplicationProperty} 
    p\wedge q\leq r \equiv p \leq q\Rightarrow r \equiv q\leq p\Rightarrow r.
\end{equation}
One of the main examples of a Heyting algebra is the set $\mathscr O(X)$ of open subsets of a topological space $X$. In this example, the implication $U\Rightarrow V$ is the union of open sets whose intersection with $U$ is contained in $V$.

\begin{lemma} \label{HeytingLemma1}
In any Heyting algebra the following hold:
\begin{enumerate}
    \item $[p\Rightarrow (q\wedge r)] = [(p\Rightarrow q)\wedge (p\Rightarrow r)]$
    \item $[(p\wedge q) \Rightarrow r] = [p\Rightarrow (q\Rightarrow r)]$
    \item $p\leq q$ if, and only if, $p\Rightarrow q = \top$
\end{enumerate}
\end{lemma}
\begin{proof} 
    For the first statement, the functor $p\Rightarrow (-)$ is a right adjoint, hence preserves limits. For the second, note that products are associative. Finally, from the adjoint property \ref{HeytingImplicationProperty} above
        \[  p = \top\wedge p \leq q \equiv \top \leq p \Rightarrow q
        \]
    as required. 
\end{proof}

Denote by `$\neg p$' the element $p \Rightarrow \bot$. This is the \textbf{negation} or \textbf{pseudocomplement} of $p$. It is the largest element of $\mathcal H$ whose meet with $p$ is $\bot$.

\begin{lemma} \label{HeytingLemma2}
For any elements $p$ and $q$ in a Heyting algebra,
\begin{enumerate}
    \item $p\leq \neg \neg p$  
    \item $p\Rightarrow \neg\neg p = \top$
    \item $p\wedge \neg p =\bot$
    \item if $p\wedge q = \bot$ then $q\leq \neg p$.
\end{enumerate}
\end{lemma}
\begin{proof} 
    For the first statement, use the fact that $\bot = p\wedge (p\Rightarrow \bot)$ and the adjoint property \ref{HeytingImplicationProperty}. The rest follow from this statement and the equivalences in \ref{HeytingImplicationProperty}.
\end{proof}

In a given Heyting algebra, generally $p \vee \neg p = \top$ does not hold. Take as an example $\mathscr O(X)$, the frame of open sets of a topological space with $X = \mathbb R$. The pseudocomplement of $\mathbb R\setminus\lbrace 0\rbrace$ is empty, so the double pseudocomplement is $\mathbb R$. Likewise the inequality $\neg (p \wedge q) \leq \neg p \vee \neg q$ is generally strict, meaning that Heyting algebras do not in general satisfy both of the DeMorgan laws.

\begin{theo}
    In a Heyting algebra, the following are equivalent:
    \begin{enumerate}
        \item $\neg\neg p \leq p$
        \item $\neg\neg p = p$
        \item $p\vee \neg p = \top$
    \end{enumerate}
    Any Heyting algebra satisfying any of these is a \textbf{Boolean algebra}.
\end{theo}
\begin{proof} 
    There are standard arguments for these equivalences; see for example the proof of Proposition I.8.4 of \cite{MM}.
\end{proof}

Heyting algebras are algebraic models of systems of propositional logic that do not obey double negation. In this sense, Heyting algebras are models of ``non-classical'' logic that is also said to be ``intuitionistic.'' Throughout given proofs will avoid using classical reasoning wherever possible. To make these argument easier to read, use some meta-language symbols: `$\supset$' stands for material implication, `$\equiv$' is biequivalence or biconditional, $\&$ is conjunction, and `$\sim$' stands for negation. Given statements $P$ and $Q$, use freely various principles of intuitionistic reasoning: namely, from $P\supset Q$ infer that $\sim Q\supset \sim P$ (contraposition); from $P$ infer $\sim\sim P$; finally that $P \supset (Q\supset \bot)$ will hold if, and only if, to $P\& Q\supset \bot$.

\section{Estimate Safety for Heyting-Valued Consensus Protocols}
\label{section:HeytingValuedSafety}

The main definition of the Abstract CBC template in \cite{AbstractCBC} can now be phrased in terms of an arbitrary Heyting algebra.

\begin{define} \label{define:HeytingValuedEstimateConsensusProtocol}
    A \textbf{Heyting-valued estimate consensus protocol} consists of
        \begin{enumerate}
            \item a set $C$ of the possible \textbf{consensus values};
            \item a Heyting algebra $\mathcal H_C$;
            \item a category $\Sigma$ whose objects $w$ are \textbf{protocol states} and whose morphisms are \textbf{protocol executions};
            \item finally an \textbf{estimator functor} $E\colon \Sigma \to \mathcal H_C$ assigning a proposition to each object of $\Sigma$ in such a way that if $Ew \Rightarrow p =\top$, then $\neg (Ew \Rightarrow \neg p)=\top$ for any proposition $p$ and any state $w$.
        \end{enumerate}
    Display a Heyting-valued estimate safety consensus protocol as $(C, \mathcal H_C, \Sigma, E)$. For the most part, work will be done for $\mathcal{H}_C = PC$, the powerset of $C$. For emphasis, call this a \textbf{Boolean-valued estimate consensus protocol}. However, in the abstract setting $PC$, will be an internally complete Heyting algebra.
\end{define}

\begin{define}
    A proposition $p$ has \textbf{estimate safety in the protocol state} $w$ if for any execution $w \to w'$, it follows that $Ew'\Rightarrow p = \top$ holds. Write `$S(p, w)$' to indicate that $p$ is safe in $w$. States $w_1$ and $w_2$ are \textbf{compatible} if there is a third state $w_3$ and executions $w_1\to w_3 \leftarrow w_2$.
\end{define}

With the definitions stated, the results of \cite{AbstractCBC} can be reproved in the present framework. Note throughout that neither double negation nor DeMorgan are required.

\begin{lemma}[Persistence Lemma] \label{PersistenceLemma}
    If $p\Rightarrow q = \top$, then $S(p,w )$ implies $S(q,w)$ for any state $w$.
\end{lemma}
\begin{proof} 
    Take an execution $w \to w'$ and assume that $Ew' \Rightarrow p = \top$. Then by Lemma \ref{HeytingLemma1}, equivalently $Ew' \leq p\leq q$ and therefore by transitivity of `$\leq$' and the same lemma $Ew' \Rightarrow q = \top$, as required.
\end{proof}

\begin{lemma}[Forward Consistency] \label{forwardconsistencylemma} 
    For any execution $w \to w'$, if $p$ is safe in $w$, then $p$ is safe in $w'$. That is, if $S(p,w)$ then $S(p,w')$.
\end{lemma}
\begin{proof} 
    Take any execution $w'\to w''$. Compose to get one $w\to w''$. Then by the hypothesis that $p$ is safe in $w$, the implication $Ew''\Rightarrow p$ holds. 
\end{proof}

\begin{lemma}[Current Consistency] \label{currentconsistencylemma} 
    If $p$ is safe in $w$, then $\neg p$ is not safe in $w$. That is, if $S(p,w)$ then $\sim S(\neg p,w)$.
\end{lemma}
\begin{proof} 
    If $S(\neg p, w)$ holds, then in particular $Ew \Rightarrow \neg p$ does too. Consequently, if $\neg(Ew\Rightarrow \neg p)$ holds, then not $S(\neg p, w)$. In general if $Ew \Rightarrow p=\top$, then $\neg(Ew\Rightarrow \neg p)\top$ by assumption on $E$. Therefore, since $S(p,w)$ implies in particular that $Ew\Rightarrow p$, the conclusion follows.
 \end{proof}

\begin{lemma}[Backward Consistency] \label{backwardconsistencylemma} 
    For all executions $w \to w'$, if $p$ is safe in $w'$, then $\neg p$ is not safe in $w$. That is, if $S(p,w')$ holds, then $\sim S(\neg p,w)$.
\end{lemma}
\begin{proof} 
    By Lemmas \ref{forwardconsistencylemma} and \ref{currentconsistencylemma}, there are valid implications
        \[ S(\neg p, w) \supset S(\neg p,w') \supset \sim S(\neg\neg p,w').
        \] 
    Therefore, by contraposition, 
        \[ \sim \sim S(\neg \neg p, w')\supset \sim S(\neg p,w).
        \]  
    But note that 
        \[ S(p,w')\supset S(\neg\neg p, w') \supset\sim\sim S(\neg \neg p, w')
        \]
    always holds intuitionistically. Putting together the last and penultimate lines, the desired result follows. 
\end{proof}

\begin{theo} \label{THEOREM1}
    Contradictory propositions are not safe at compatible states. That is, if $p\wedge q = \bot$ and $w_1\simeq w_2$ both hold, then $\neg (S(p,w_1)\wedge (S(q, w_2))$ holds. 
\end{theo}
\begin{proof} 
    By Lemma \ref{HeytingLemma2}, $p\wedge q = \bot$ is equivalent to $q\leq \neg p$ which is equivalent to $q\Rightarrow \neg p=\top$. Now, by forward consistency and backward consistency, the implications
        \begin{equation} \label{equation1}
            S(p, w_1)\supset S(p, w_3) \supset \sim S(\neg p,w_2)
        \end{equation} 
    hold. By Lemma \ref{PersistenceLemma} applied to $q\Rightarrow \neg p$,
        \[ 
            S(q, w_2)\supset S(\neg p, w_2)
        \]
    and intuitionistic contraposition,
        \begin{equation} \label{equation2}
            \sim S(\neg p, w_2) \supset \sim S(q, w_2).
        \end{equation} 
    Therefore, putting together the implications \ref{equation1} and \ref{equation2},
        \[ 
            S(p, w_1) \supset \sim S(q, w_2)
        \]
    holds. But, by definition of negation, the last display is equivalently
        \[  
            S(p, w_1) \supset (S(q, w_2) \supset \bot)
        \]
    which is intuitionistically equivalent to 
        \[ 
            (S(p, w_1) \, \& \, S(q, w_2)) \supset \bot
        \]
    by Lemma \ref{HeytingLemma1} that is, to $\sim (S(p, w_1) \, \& \, S(q, w_2))$, as required.
\end{proof}

\section{Toposes and Their Internal Logic}
\label{section:ToposesInternalLogic}

The forcing interpretation of safety in a topos requires some background and notation on topos theory. References are the standard ones, such as \cite{MM} and \cite{ElephantI}. 

Recall that a \textbf{subobject classifier} in a finitely-complete category $\E$ is a morphism $\top\colon  1\to \Omega$ such that for any subobject $m\colon S\to X$ there is a unique characteristic map $\chi_m\colon X\to \Omega$ making a pullback
$$\xymatrix{
    S \ar[d]_m \ar[r] & 1 \ar[d]^\top \\
    X \ar[r]_{\chi_m} & \Omega.
}$$
In non-elementary terms, this is to say that $\sub(-)\colon \E^{op}\to\set$ is representable, that is, pulling back along $X\to\Omega$ induces isomorphisms $\sub(X)\cong \E(X,\Omega)$ holding naturally in $X$. A morphism $X\to \Omega$ in a topos is a \textbf{propositon} ranging over the elements of $X$. A special proposition is the composite of the unique arrow $X\to 1$ and $\top \colon 1\to \Omega$ denoted throughout by `$\top_X$'. Think of this as ``true with respect to $X$.''

\begin{define}
A \textbf{topos} is a finitely complete category $\E$ with power objects and a subobject classifier. 
\end{define}

\begin{example}
    The category of sets $\set$, the category of finite sets $\mathbf{Fin}$, any presheaf category $[\C^{op},\set]$, any sheaf category $\mathbf{Sh}(\C,J)$, and any arrow category $\E^\mathbf 2$ on a topos $\E$ are all toposes.The two-element set $\mathbf 2 = \lbrace 0,1\rbrace$ is a subobject classifier in $\set$. The presheaf $\C^{op}\to \set$ making the assignment 
        \[
            C\mapsto \lbrace \text{sieves on } C\rbrace
        \]
    is a subobject classifier in the presheaf category $[\C^{op},\set]$. The top element of $\Omega(C)$ is the so-called ``total sieve'' on $\C$, namely, the collection $\mfrk t_C = \lbrace f\colon D\to C\rbrace$ consisting of all the arrows of $\C$ with codomain $C$. The classifying arrow $\chi$ associated to a subobject $S\to X$ in $[\C^{op},\set]$ has as its components $\chi_C\colon XC\to \Omega C$ the functions
        \begin{equation}
            x \mapsto \lbrace f\colon D\to C \mid f^*x\in SD\rbrace
        \end{equation}
    where $f^*\colon FD\to FC$ denotes the associated transition function $F(f) = f^*$.
    \end{example}

The \textbf{power object} $PX$ of a topos object $X$ comes with a \textbf{membership} morphism $\in_X \colon  X\times PX \to \Omega$ having the universal property that for any morphism $g\colon X\times Y\to \Omega$, there is a unique \textbf{transpose} $\hat g\colon Y\to PX$ such that $g = \in_X (1\times \hat g)$ holds. Consequently, propositions $p\colon X\to\Omega$ are in bijection with global elements $\hat p\colon 1\to PX$. Each such power object is an internal frame. For each subobject lattice $\sub(C)$ is one and by the naturality of the isomorphisms
    \[
        \sub(C\times X)\cong \E(X,PC)
    \]
the Heyting operations on $\sub(C\times X)$ induce operations on $PC$ by Yoneda. The internal ordering relation $(\leq)$ can be described as the equalizer
$$\xymatrix{
    (\leq) \ar[r] & PX\times PX \ar@<.5ex>[r]^-{\pi_1} \ar@<-.5ex>[r]_-\wedge & PX.
}$$
Write `$f\leq g$' if the pair $\langle f,g\rangle$ factors through $(\leq)$. For the special case of $\Omega = P1$, the isomorphism above reduces to 
    \[
        \sub(X)\cong \E(X,\Omega)
    \]
which is the unique frame isomorphism making $\Omega$ an internal frame. The classifying map of the internal order object $(\leq)\to \Omega\times\Omega\to \Omega$ is the implication operator $\Rightarrow\colon \Omega\times\Omega \to \Omega$. 

\section{Geometric Models}
\label{section:GeometricModels}

Morphisms of toposes give rise to one version of the forcing semantics presented here. First recall the standard definition and some conventions.

\begin{define}
    A \textbf{geometric morphism} $F\colon \F\to \E$ between toposes is a pair of adjoint functors $F^*\dashv F_*$ where $F^*$ is finite limit preserving. Call $F^*\colon \E\to \F$ the \textbf{inverse image} and $F_*\colon \F\to\E$ the \textbf{direct image}.
\end{define}

Conventionally geometric morphisms point in the direction of their direct image. For any geometric morphism $F\colon \F\to\E$, denote by $A^*$ the action of $F^*$ on $A\in\E$, namely, $F^*A=A^*$. Similarly, $X_*$ for $X\in\F$ denotes $F_*X$ in $\E$. Denote by $\overline f$ the \textbf{transpose} in $\F$ of a morphism $f\colon A \to X_*$, that is, $\overline f = \epsilon f^*$. Likewise, $\hat g$ denotes the transpose in $\E$  of a morphism $g\colon A^*\to X$ in $\F$, namely, $g_*\eta$.

\begin{example}
Order-preserving morphisms $f\colon P \rightleftarrows Q:g$ between posets internal to a topos are \textbf{internally adjoint} with $f\dashv g$ if $fg\leq 1$ and $1\leq gf$ both hold. Since $P\colon \E^{op}\to\E$ is a functor, each arrow $f\colon A\to B$ induces one $Pf\colon PB\to PA$. Each such arrow has an internal left adjoint $\exists_f$ and an internal right adjoint $\forall_f$, similiarly induced from external adjoints on subobject lattices by the Yoneda isomorphism. In the special case of the morphism $C\to 1$ this situation is summarized by 
    $$\xymatrix{
        PC \ar@<1.5ex>[rr]^-{\forall_C} \ar@<-1.5ex>[rr]_{\exists_C} &  & \ar[ll]|{\Delta_C} \Omega  & \exists_C\dashv \Delta_C\dashv \forall_C.
    }$$
These adjoints compose $\Box := \Delta\forall$ and $\Diamond := \Delta\exists$ yielding an ``adjoint modality" $\Diamond \dashv \Box$. Consequently, any power object in a topos is an ``internal S4-modal algebra." 
\end{example}

\begin{example} \label{exampleAlgebrasInducedbyGeometricMorphisms}
Let $F\colon \F\to \E$ denote a geometric morphism between toposes. The direct image of the subobject classifier $F_*\Omega_\F$ is again a complete Heyting algebra in $\E$. Since $\Omega_\E$ is the initial frame in $\E$, there is a unique frame morphism $i\colon \Omega_\E\to F_*\Omega_\F$. On the other hand, a morphism $\tau \colon F_*\Omega_\F\to \Omega_\E$ classifies the top element of $F_*\Omega_\F$. These morphisms are internally adjoint with $i\dashv \tau$ and by the proposition, $i\tau = \Box$ makes $f_*\Omega_\F$ into an S4-modal algebra in $\E$ (cf. Lemmas 1.2 and 1.3 of \cite{Steve}).\end{example} 

As observed in the proof of Proposition 4.2 of \cite{Steve}, the adjunction $i\dashv \tau$ from the example arises via the Yoneda Principle from an external adjunction $\Delta \dashv \Gamma$ natural in $A$ and various natural isomorphisms
    $$\xymatrix{
        \E(A,\Omega_\E) \cong \sub_\E(A) \ar@<-1ex>[rr]_-\Delta & & \ar@<-1ex>[ll]_-\Gamma \sub_\F(f^*A)\cong \F(f^*A,\Omega_\F) \cong \E(A,\Omega_*).
    }$$
See the reference for the construction of $\Gamma$. However, $\Delta$ is the restiction of $F^*$ to subobjects and that it is injective implies that the unique frame homomorphism $i$ is monic. The converse is also true.

\begin{prop}
    The following are equivalent:
        \begin{enumerate}
            \item $\Delta$ above is injective;
            \item $f^*$ is faithful;
            \item $i\colon \Omega_\E\to\Omega_*$ is monic.
        \end{enumerate}
    In the event that any of these conditions are satsified, $F$ is said to be a \textbf{surjection}.
\end{prop}
\begin{proof} See Lemma VII.4.3 of \cite{MM} for the equivalence of the first two statements. That the third condition implies the first is a consequence of Yoneda.     
\end{proof}

\begin{define}
     A surjective geometric morphism $F\colon \F\to \E$ is a \textbf{geometric model}. The associated operator $\Box$ is a \textbf{geometric modality}.
\end{define}

\begin{remark} The typical situation is the following. Let $F\colon \C\to \D$ denote any functor. In the development below, this will be an estimator $E\colon \Sigma\to PC$ coming with an Abstract CBC protocol as in Definition \ref{define:HeytingValuedEstimateConsensusProtocol}. In the general case, there is an induced essential geometric morphism
    $$\xymatrix{
        [\C,\set] \ar@<1.5ex>[rr]^-{F_*} \ar@<-1.5ex>[rr]_{F_!} &  & \ar[ll]|{F^*} [\D,\set]  & F_!\dashv F^* \dashv F_*
    }$$
where the adjoints to substitution are given by left and right Kan extensions. In this case, $F_*\Omega_\C$ is an S4-modal algebra in $[\D,\set]$. This is a surjection if, and only if, every object of $\D$ is a retract of on in the image of $F$ (cf. A4.2.7(b) of \cite{ElephantI}). Thus, an estimator $E\colon \Sigma\to PC$ in an Abstract CBC protocol may give rise to a geometric model even if it is not an epimorphism by taking the bo-ff factorization of $e$ and forgetting the ff-part. So, it might as well be assumed that $e$ is surjective on objects. However, this is an unnatural assumption as discussed in the Introduction.
\end{remark}

The semantics of the operator $\Box$ originate with \cite{KripkeModal}. The basis of the present development is proposition 4.9 of \cite{Steve} that gives a forcing semantics of geometric modalities induced by surjective geometric morphisms. The following preliminary result will be needed later on.

\begin{lemma} \label{LemmaRestrictToIdentity}
    If $F\colon \F\to \E$ is a geometric model, then for any $\phi\colon A\to \Omega_*$ in $\E$, the transposes satisfy $\overline{\Box\phi} =\top_A$ if, and only if, $\overline{\phi} = \top_A$.
\end{lemma}
\begin{proof}
    Consider the diagram
        $$\xymatrix{
            & 1 \ar[d]_{\top_*} \ar@{=}[r]  & 1 \ar[d]^\top \ar@{=}[r] & 1 \ar[d]^{\top_*} \\
            A \ar[r]_\phi &  \Omega_* \ar[r]_\tau  & \Omega_\E \ar[r]_i & \Omega_*
        }$$
    Since $i$ is a frame homomorphism, the rightmost square commutes; since $i$ is monic, it is a pullback. Thus, $\phi\colon A\to \Omega_*$ factors through $\top_*$ if, and only if, $\Box\phi$ factors through $\top_*$. Since $\epsilon$ is natural, the square
        $$\xymatrix{
            1 \ar[d]_{(\top_*)^*} \ar@{=}[r] & 1 \ar[d]^\top \\
            (\Omega_*)^* \ar[r]_-\epsilon & \Omega_\F 
        }$$
    commutes. Thus, for any $\psi \colon X\to\Omega_*$, the equation $\overline \psi = \top_X$ holds if, and only if, $\psi$ factors through $\top_*$ in $\E$. Therefore, the equivalences
    \begin{align}
        \overline\phi = \top_A &\equiv \phi\text{ factors through } \top_*\notag \\
                        &\equiv \Box \phi \text{ factors through } \top_*\notag \\
                        &\equiv \overline{\Box\phi} = \top_A\notag
    \end{align}
    establish the result.
\end{proof}

\section{Presheaf Forcing Semantics of Estimate Safety}
\label{section:ForcingSemanticsofEstimateSafety}

Forcing originates with Kripke's semantics for intuitionistic logic \cite{Kripke}. For further background on Kripke-Joyal forcing as a semantics of topos logic see Chapter VI of \cite{MM}. This section illustrates the role of the total cosieve in the forcing semantics of safety. While this does require the use of geometric models, it forms the basis of the elementary axiomatization and forcing results later on.

\begin{define}
    A morphism $a\colon W\to X$ \textbf{forces} a proposition $\phi\colon X\to \Omega$ if the image of $a\colon W\to X$ factors through $S_\phi$ as in the diagram
        $$\xymatrix{
        & S_\phi \ar[d] \ar[r]  & 1 \ar[d]^\top  \\
        \mathrm{Im}(a) \ar@{-->}[ur] \ar[r] &  X \ar[r]_\phi  & \Omega
        }$$
    Equivalently, $a$ forces $\phi(x)$ if $\phi(a) = \top_W$. Denote this situation by $W\Vdash \phi(a)$.
\end{define}

\begin{example}
    In the special case where $\E$ is a presheaf topos, the forcing relation has an especially nice form. It suffices to restrict to a generating set, namely, that consisting of the canonical representable functors and to genuine elements $a\in X(C)$. One writes `$C \Vdash \phi(a)$' in the place of `$\mathbf yC\Vdash \phi(a)$.' The forcing relation takes the form
        \[
            C\Vdash \phi(a) \text{ if, and only if, } a\in S_\phi \text{ if, and only if, } \phi_C(a) = \mfrk t_C
        \]
    That is, $a$ forces $\phi$ in stage $C$ if, and only if, $a$ is in the comprhension of $\phi$ at stage $C$ if, and only if, $\phi_C$ evaluates at $a$ to the total sieve on $C$.
\end{example}

\begin{lemma}[Stability] \label{lemma:ForcingStability}
    If $W\Vdash \phi(a)$, then $V\Vdash\phi(ab)$ holds for any $b\colon V\to W$.  
\end{lemma}
\begin{proof}
    This follows by the uniqueness of image factorizations in a topos.
\end{proof}

Presheaf forcing as in the example can be relativised to an over-topos using a geometric morphism. Consider the adjoint situation
    $$\xymatrix{
        [\C,\set] \ar@<1.5ex>[rr]^-{F_*} \ar@<-1.5ex>[rr]_{F_!} &  & \ar[ll]|{F^*} [\D,\set]  & F_!\dashv F^* \dashv F_*
    }$$
Recall that the transpose of a proposition is $\overline\phi = \epsilon \phi^*$, obtained by applying $F^*$ and composing with the counit $\epsilon$. The following gives the definition of the relativised notion of forcing. It is a modification of Definition 4.8 in \cite{Steve} by allowing $F$ to be an arbitrary functor $\C\to\D$.

\begin{define}[Relativised Forcing] \label{define:RelForcing}
    An element $a\in X(FC)$ \textbf{forces} $\phi(x)\colon X\to \Omega_*$ in state $C$ if $C \Vdash \overline\phi(a)$ holds in $[\C,\set]$. Denote this situation by $C\Vdash_* \phi(a)$.
\end{define}

The required semamtics for this relativisted notion of forcing is the following. Notice that the proof does not require that $F$ induces a geometric model.

 \begin{lemma}[Presheaf Forcing Semantics] \label{lemma:SemanticsofTotalCosieve}
    $C\Vdash_* \phi(a)$ holds if, and only if, $D\Vdash_* \phi(f_!a)$ for all $f\colon C\to D$.
\end{lemma}
\begin{proof}
    The equivalences 
        \begin{align}
            C\Vdash_* \phi(a) & \equiv C\Vdash \epsilon \phi^*(a) \qquad &\text{(def. `$\Vdash_*$')}\notag \\
                             & \equiv \epsilon \phi^*(a) = \top_C \qquad &\text{(def. `$\Vdash$')} \notag \\
                             & \equiv \lbrace f\colon C\to D\mid f_!(a) \in S_{\epsilon\phi^*}\rbrace = \mfrk t_C \qquad &\text{(constr. `$\chi$')} \notag \\
                             & \equiv f_!(a) \in S_{\epsilon\phi^*} \text{ for all } f\colon C\to D \qquad & \text{(def. $\mfrk t_C$)} \notag\\
                             & \equiv \epsilon\phi^*(f_!(a)) = \top_D \text{ for all } f\colon C\to D \qquad &\text{(def. $S_{\epsilon\phi^*}$)}\notag\\
                             & \equiv D\Vdash \epsilon\phi^*(f_!(a)) \text{ for all } f\colon C\to D \qquad &\text{(def. `$\Vdash$')}\notag\\
                             & \equiv D\Vdash_*\phi(f_!(a)) \text{ for all } f\colon C\to D \qquad &\text{(def. `$\Vdash_*$')}\notag
        \end{align}
    establish the result.
\end{proof}

Now, give a presheaf forcing semantics of estimate safety. Fix the estimator $E\colon \Sigma_0\to PC$ in an estimate consensus protocol as in Definition \ref{define:HeytingValuedEstimateConsensusProtocol} and consider the associated geomtric morphism $e\colon \F\to\E$ with $\F= [\Sigma,\set]$ and $\E= [PC,\set]$. Interpret a proposition $p\in PC$ as a subobject of $1$ in $[PC,\set]$ by 
\begin{equation} \label{inducedpoint0}
    p(S) = \begin{cases} 1 \qquad &\text{if $p\subset S$} \\ 0 \qquad &\text{otherwise} \end{cases} 
\end{equation}
This is the same as noting that $p$ induces a canonical representable functor $yp\colon PC\to \set$ and then by taking support $yp \to U \to 1$ via image factorization there is a corresponding subobject of $1$. Likewise, each $w\in \Sigma$ determines two such subojects -- one in $[\Sigma,\set]$ in the same manner, and another in $[PC,\set]$ since $ew$ is an element of $PC$. In the theorem below, interpret such states and propositions used in formulas as subobjects of $1$. For a proposition $p\colon U\to 1$, let $\chi_p\colon 1\to \Omega$ denote the classifying arrow and let $i\colon \Omega\to \Omega_*$ denote the unique frame homomorphism. This is monic if, and only if, $E$ induces a geometric model. In this case $p$ is the subobject classified by $i\chi_p$ since the square involving $i$ is also a pullback.

\begin{theo} \label{theo:PresheafForcingSemanticsofSafety}
    Suppose that $E\colon \Sigma_0\to PC$ induces a geometric model. A proposition $p\colon U\to 1$ is safe in state $w$ if, and only if, $w\Vdash_* i\chi_p$.
\end{theo}
\begin{proof}
    Consider the following equivalences:
        \begin{align}
            w\Vdash_* i\chi_p & \equiv v\Vdash_* i\chi_p \text{ for all } w\to v \qquad &\text{(Lemma \ref{lemma:SemanticsofTotalCosieve})} \notag \\
                              & \equiv v \Vdash \overline{i\chi_p} \text{ for all } w\to v \qquad &\text{(def `$\Vdash_*$')} \notag \\
                              & \equiv v\leq S_{\,\overline{i\chi_p}} \text{ in } \sub_\F(1) \text{ for all } w\to v \qquad &\text{(def. `$\Vdash$')} \notag \\
                              & \equiv ev\leq p \text{ in } \sub(1)_\E \text{ for all } w\to v \qquad &\text{(transpose \& $i$ monic)} \notag \\                     
                              & \equiv ev \Rightarrow p=\top  \text{ in } \sub(1)_\E \text{ for all } w\to v.\qquad &\text{(Lemma \ref{HeytingLemma1})} \notag
        \end{align}
    Note that the second to last step uses the fact that the transpose of $S_{\,\overline {i\chi_p}}$ is isomorphic to $p$ as subobjects of $1$ in $\E$ because $p$ is the subobject classified by $i\chi_p$ since $i$ is monic.
\end{proof}

\begin{cor}[Persistence]
    If $E$ induces a geometric model, then if $p$ is safe in $w$ and $w\to v$ is an execution, then $p$ is safe in $v$ too.   
\end{cor}
\begin{proof}
    The theorem shows that safety is a forcing relation. Lemma \ref{lemma:SemanticsofTotalCosieve} then establishes the statement since any state accessible from $v$ is one accessible from $w$.
\end{proof}

Now, give a modal interpretation of estimate safety. Adopt the same set-up as for Theorem \ref{theo:PresheafForcingSemanticsofSafety} above. The required semantics for $\Box$ extends Proposition 4.9 in \cite{Steve} from inclusions $|\C|\to\C$ to functors $\C\to\D$ inducing a geometric model. 

\begin{theo}[$\Box$-Semantics] \label{theorem:BoxSemanticsElementaryVersion} Suppose that $F\colon \F\to \E$ is a geometric model. Then $C\Vdash_*\Box \phi(a)$ holds if, and only if, $D\Vdash_* \phi(f_!a)$ holds for all $f\colon C\to D$.
\end{theo}
\begin{proof} Consider the following equivalences:
    \begin{align}
        C\Vdash_*\Box\phi(a) & \equiv C\Vdash \overline{\Box\phi}(a) \notag \qquad &\text{(def. `$\Vdash_*$')} \\
                             & \equiv \overline{\Box\phi}(a) = \mfrk t_C \notag \qquad &\text{(def. `$\Vdash$')} \\
                             & \equiv \overline{\phi}(a) = \mfrk t_C \notag \qquad &\text{(Lemma \ref{LemmaRestrictToIdentity})} \\
                             & \equiv C\Vdash \overline{\phi}(a) \notag \qquad &\text{(def. `$\Vdash$')} \\
                             & \equiv C\Vdash_*\phi(a) \notag \qquad &\text{(def. `$\Vdash_*$')} \\
                             & \equiv D\Vdash_* \phi(f_!a) \text{ for all } f\colon C\to D \qquad &\text{(Lemma \ref{lemma:SemanticsofTotalCosieve})} \notag
    \end{align}
    These establish the result.
\end{proof}

\begin{cor}[Modal Interpretation of Safety] \label{cor:PresheafSafetyisModal}  Assume that $E$ results in a geometric model. A proposition $p$ is safe in $w$ if, and only if, $w\Vdash_* \Box i\chi_p$.
\end{cor}
\begin{proof}
    This follows from Theorem \ref{theo:PresheafForcingSemanticsofSafety} above, Lemma \ref{lemma:SemanticsofTotalCosieve} and finally Theorem \ref{theorem:BoxSemanticsElementaryVersion}.
\end{proof}

\begin{remark} \label{remark:SheafStuff}
    There are at least two reasons that it might be expected the safety results above would involve passing to sheaf toposes. First any $PC$, even in a elementary topos, is the direct image of the subobject classifier of a sheaf topos, namely, sheaves on $PC$ viewed as an internal frame (cf. \S 5 of \cite{Steve}). Additionally the space of consensus values should be nonempty, that is, should admit an epimorphism $C\to 1$, meaning that the induced geometric morphism from sheaves on $PC$ to the base topos should be a surjection. Clearly, $\Omega$ in the base topos then embeds faithfully into $\Omega_*$ and the conditions of the previous results are satisfied. Additionally, subobjects of $1$ in sheaves on $PC$ are precisely the global elements of $PC$, so there is a 1-1 correspondence between propositions about consensus values and subobjects of $1$. However, this does not work for at least a couple of reasons. First is that forcing should happen relative to protocol states, not generalized elements of $PC$. Secondly, the semantics of the total coseive in Lemma \ref{lemma:SemanticsofTotalCosieve} were crucial in the sense that it ensures that the directed upperset on a state $w$ is contained below the safe proposition $p$. This be a cosieve, as otherwise, for ordinary sheaves (contravariant functors!), the subobject classifier is sets of sieves, namely, \emph{downward closed} sets.
\end{remark}

\section{Consensus Protocols in a Topos}
\label{section:ConsensusProtocolsinaTopos}

Abstract CBC can be formulated in an arbitrary topos. An estimator will pick out a proposition for every protocol state in a functorial way. Therefore, an estimator is an internal functor $e\colon \Sigma\to PC$ satisfying the internalized version of the compatibility condition. Here is the formal definition.

\begin{define} \label{define:abstractCBC}
An \textbf{estimate consensus protocol} in a topos $\E$ consists of
\begin{enumerate}
    \item an object $C$ of consensus values;
    \item an internal category $\Sigma$ of protocol states $\Sigma_0$ and executions $\Sigma_1$;
    \item an internal functor $e\colon \Sigma\to PC$ called the \textbf{estimator} satisfying the condition that  
        \begin{equation} \label{equation:EstimatorCondition}
            \text{ if } e(w) \Rightarrow p = \top \text{ for some state } w\colon X\to \Sigma_0, \text{ then } \neg (e(w)\Rightarrow \neg p) = \top
        \end{equation}
        for any proposition $p\colon 1\to PC$.
\end{enumerate}
\end{define}

This phrasing has the advantage of incorporating all the data in a single morphism $e\colon \Sigma \to PC$ of $\cat(\E)$. It therefore is not too much to identify an estimate safety consensus protocol with the internal functor $e\colon \Sigma \to PC$. Now, the main definition of the paper:

\begin{define} \label{definitionToposSafety}
A proposition $p\colon 1\to PC$ is \textbf{safe in the protocol state} $w\colon W\to \Sigma_0$ if for any execution $f\colon w \to v$, it follows that $ew\Rightarrow p = \top$ holds. 
\end{define}

\begin{prop}[Persistence] \label{PersistenceLemma2}
    If $p\Rightarrow q = \top$, then that $p$ is safe in $w$, implies that $q$ is safe in $w$ too.
\end{prop}
\begin{proof} Start with any execution $w\to v$ and assume that $ev \Rightarrow p = \top$. Then by Lemma \ref{HeytingLemma1}, equivalently $ev \leq p\leq q$ holds and therefore by transitivity of `$\leq$' and the same lemma $ev \Rightarrow q = \top$, as required. \end{proof}

\begin{define}
States $w_1, w_2:\Sigma_0$ are \textbf{compatible} if there is a state $w_3:\Sigma_0$ and executions $w_1\to w_3 \leftarrow w_2$. Denote compatibility by `$w_1\simeq w_2$'. Compatible states are said to have a \textbf{common future}.
\end{define}

In this set-up it is possible to give a completely elementary account of the presheaf forcing semantics of the previous sections. This is a prelude to the account given in \S \ref{section:ElementarySafetyInaTopos} that uses internal category theory without direct reference to the internal diagram categories used here. This is a somewhat tangential discussion that can be skipped without losing the main thread. 

First recall some standard definitions (cf. \S V.7 of \cite{MM} or \S B2.3 of \cite{ElephantI}). Throughout let $\basetopos$ denote a topos thought of as playing the role of $\set$. An \textbf{internal category} $\bb C$ in $\basetopos$ consists of the data of objects and arrows
    $$\xymatrix{
        C_1\times_{C_0}C_1 \ar[r]^-\otimes & C_1 \ar@<1.1ex>[r]^{d_0} \ar@<-1.1ex>[r]_{d_1} & \ar[l]|{y} C_0
    }$$
satisfying the usual axioms for a category in diagrammatic form. An \textbf{internal functor} $f\colon \bb C \to \bb D$ consists of arrows $f_0\colon C_0\to D_0$ and $f_1\colon C_1\to D_1$ commuting with the identity, domain, codomain and composition morphisms coming with $\bb C$ and $\bb D$.

\begin{define}
    Let $\bb C$ denote a category in $\basetopos$. An \textbf{internal diagram} on $\bb C$ consists of an arrow $\gamma\colon X\to C_0$ and an action morphism $m\colon X\times_{C_0}C_1\to C_1$ such that 
        $$\xymatrix{
            X\times_{C_0}C_1\times_{C_0}C_1 \ar[d]_{m\times 1} \ar[r]^-{1\times m} & X\times_{C_0}C_1 \ar[d]^m & & X \ar@{=}[dr] \ar[r]^-{\langle \gamma,y\rangle} & X\times_{C_0}C_1 \ar[d]^m \\ 
            X\times_{C_0}C_1 \ar[r]_-m & X & & & X
        }$$
    both commute. A morphism of internal diagrams is an arrow $f\colon X\to Y$ that is equivariant with respect to the actions; that is, the diagram
        $$\xymatrix{
            X\times_{C_0}C_1 \ar[d]_m \ar[r]^-{f\times 1} & Y\times_{C_0}C_1 \ar[d]^n \\
            X \ar[r]_f & Y
        }$$
    commutes. Internal diagrams and their morphisms form a category $[\bb C,\basetopos]$.
\end{define}

Categories of internal diagrams are an elementary version of base-valued functors, that is, ordinary presheaves on the given category. Needed results on such categories are developed in \S B2.3 of \cite{ElephantI}. There it is shown that the underlying functor $[\bb C,\basetopos] \to\basetopos/C_0$ is comonadic, making $[\bb C,\basetopos]$ a topos. Additionally, any internal functor $F\colon \bb C\to \bb D$ induces a geometric morphism 
    \[
        F\colon [\bb C,\basetopos] \to [\bb D,\basetopos].    
    \]
This is an internal analogue of the geometric morphism induced by a functor between ordinary presheaf toposes. Forcing has the following form. A proposition and generalized element will take the form of commutative triangles such as
    $$\xymatrix{
        Z \ar[d]_c \ar[r]^a & X \ar[d]_\gamma \ar[r]^\phi & \Omega \ar[d]^\omega \\ 
        C_0 \ar@{=}[r] & C_0 \ar@{=}[r] & C_0
    }$$
Mimicing the presheaf phrasing of forcing, $c\Vdash \phi(a)$ holds if $a$ factors through $S_\phi$ in $[\bb C,\basetopos]$. Denote the map accomplishing this by $\hat a \colon W \to S_\phi$ where $u\colon S_\phi\to X$ computes the required pullback in $[\bb C,\basetopos]$. The first result concerning forcing follows essentially by equivariance. It is an internal version of Lemma \ref{lemma:SemanticsofTotalCosieve}.

\begin{lemma} \label{lemma:InternalTechnicalResult}
    For any $\phi\colon X\to \Omega$ in $[\bb C,\basetopos]$, if $c\Vdash \phi(a)$ holds, then $d\Vdash\phi(f_!a)$ for all $f\colon c\to d$ with $f\colon C_1$.
\end{lemma}
\begin{proof} 
    Here $f_!(a) = m\langle a,f\rangle$ is the result of the action of the element $f$ of $C_1$ on $a$. It needs to be seen that $\phi f_!(a)$ factors through $S_\phi$ as well. This is proved by the diagram (ignoring the morphisms to the base $C_0$) 
        $$\xymatrix{
            & S_\phi\times_{C_0}C_1 \ar[d]_u \ar[r]^-n  & S_\phi \ar[d]^u \ar[r] & 1 \ar[d]^\top \\
            Z \ar[ur]^-{\langle\hat a,f\rangle} \ar[r]_-a &  X\times_{C_0}C_1 \ar[r]_-m  & X \ar[r]_\phi & \Omega
        }$$
    which commutes by construction of $\hat a$, construction of $S_\phi$ and finally the fact that $S_\phi$ and $X$ are both internal diagrams and that $u$ preserves the action of $C_1$.
\end{proof}

With the technical result proven, the forcing semantics of the adjoint modality can be given. First internalize Definition \ref{define:RelForcing} to an induced geometric morphism of diagram toposes.

\begin{define}
    Let $F\colon \bb C\to \bb D$ denote an internal functor in $\basetopos$. An element $a\colon W\to X$ as above \textbf{forces} $\phi\colon X\to \Omega_*$ in state $w$ if $w \Vdash \overline\phi(a)$ holds in $[\bb C,\basetopos]$. Denote this situation by $w\Vdash_* \phi(a)$.
\end{define}

\begin{theo}[Internalized Forcing Semantics] \label{theorem:InternalForcingSemantics} 
    In the notation above, $c\Vdash_* \phi(a)$ holds if, and only if, $d\Vdash_* \phi(f_!a)$ holds for all $f\colon c\to d$.
\end{theo}
\begin{proof} 
    The computation of Theorem \ref{lemma:SemanticsofTotalCosieve} can be recreated internally using Lemma \ref{lemma:InternalTechnicalResult}.
\end{proof}

\section{Estimate Safety in a Topos}
\label{section:ElementarySafetyInaTopos}

Provided one can countenance a little internal category theory, the forcing semantics of safety take on an especially nice form without requiring passage to internal diagram categories and without appealing to the mechanisms of geometric models. The point of the geometric models development was the revelation of the central of the total cosieve, whose elementary analogue can be developed with some internal category theory.

Throughout work internally with a estimate consensus protocol $e\colon \Sigma\to PC$ in a given fixed topos $\E$ as in Definition \ref{define:abstractCBC}. For any state $w\colon W\to \Sigma_0$, form the object of executions from $w$ as the pullback
    $$\xymatrix{
        \Sigma(w,-) \ar@{->>}[d] \ar[r] & \Sigma_1\ar[d]^{d_0} \\
        W \ar[r]_w & \Sigma_0.
    }$$
This will play the role of the total cosieve on $w$. Notice that since $d_0$ is an epimorphism (it is split by $i$), so is the projection to $W$, as epimorphisms are pullback-stable. In other words, $\Sigma(w,-)$ is a sort of generalized protocol execution defined on $w$. Now take any execution $f\colon w\to v$ on $w$, that is, a generalized element $f\colon X\to \Sigma_1$ with an epimorphism making a commutative square
    $$\xymatrix{
        X \ar@{->>}[d] \ar[r]^f & \Sigma_1 \ar[d]^{d_0} \\
        W \ar[r]_w & \Sigma_0.
    }$$
There is then a unique morphism $\hat f\colon X\to \Sigma(w,-)$ by the universal property of the pullback, interpretable as the statement that $f$ is an element of the ``fiber'' of the reprentable functor at $v\in \Sigma_0$. For any proposition $p\colon 1\to PC$, form the implication $x\Rightarrow p$ for a variable $x:PC$ as the composite
    $$\xymatrix{
        PC\cong PC\times 1 \ar[r]^-{x\times p} & PC\times PC \ar[r]^-{\Rightarrow} & \Omega
    }$$
where `$\Rightarrow$' is the classifying arrow of the order object $(\leq)\to PC\times PC$. There is then the following result, essentially stating that estimate safety is equivalent to forcing this proposition by the representable $\Sigma(w,-)$. As is customary with presheaf forcing, identify $w$ with the representable $\Sigma(w,-)$ in the forcing notation. That is, write `$w \Vdash x\Rightarrow p$' as a shorthand for `$\Sigma(w,-) \Vdash x\Rightarrow p$'.

\begin{theo} \label{theorem:ElementaryForcingSemanticsOfSafety}
    A proposition $p\colon 1\to PC$ is safe in state $w$ if, and only if, $w \Vdash (x\Rightarrow p)(ed_1)$.
\end{theo}
\begin{proof} First show necessity. Since $\Sigma(w,-)$ is an execution defined on $w$, the morphism $e d_1\pi_2\colon \Sigma(w,-)\to PC$ satisfies $d_1e\pi_2 \Rightarrow p = \top$ by the hypothesis of safety. Thus, there is a factorization
    $$\xymatrix{
        & S_{x\Rightarrow p} \ar[d] \ar[r]  & 1 \ar[d]^\top  \\
        \Sigma(w,-) \ar@{-->}[ur] \ar[r]_-{ed_1\pi_2} &  PC \ar[r]_{x\Rightarrow p}  & \Omega
    }$$
since the square is a pullback, proving the forcing statement. On the other hand, for sufficiency, assume the forcing statement. In particular, $ed_1\pi_2 \Rightarrow p = \top$ holds. Let $f\colon w\to v$ be any execution on $w$. Then by definition, $f$ factors through $\Sigma(w,-)$ via a unique map $\hat f\colon X \to \Sigma(w,-)$ satisfying in particular $\pi_2\hat f = f$. Thus, compute that 
    \begin{align}
        \top &= ed_1\pi_2\hat f \Rightarrow p \notag \\
             &= ed_1f \Rightarrow p \notag \\
             &= ev \Rightarrow p \notag
    \end{align}
proving that $p$ is safe in $w$.
\end{proof}

However, in light of Theorem \ref{theorem:ElementaryForcingSemanticsOfSafety}, write `$w\Vdash x\Rightarrow p$' to indicate that $p$ is safe in $w$. The practical upshot is that the well-known forcing semantics (cf. Theorem VI.6.1 of \cite{MM}) of usual logical connectives can be used to prove safety results. In particular, for any implication statement $\phi\Rightarrow \psi$ with $\phi,\psi\colon X\rightrightarrows \Omega$ and $a\colon U\to X$, there is the equivalence
    \begin{equation} \label{equation:ForcingSemanticsofImplication}
        U\Vdash \phi(a) \Rightarrow \psi(a) \text{ if, and only if, } V\Vdash \phi(af) \text{ implies } V\Vdash \psi(af) \text{ for all } f\colon V\to U.
    \end{equation}
As an immediate consequence, there is the next result.

\begin{cor}
    A proposition $p$ is safe in state $w$ if, and only if, $v\Vdash x(ev)$ implies that $v\Vdash p$ for all executions $f\colon w\to v$ on $w$.
\end{cor}
\begin{proof}
    Interpret $p\colon 1\to PC$ as defined on $PC$ by composing with the unique map $PC\to 1$. Then use Theorem \ref{theorem:ElementaryForcingSemanticsOfSafety} and the equivalence \ref{equation:ForcingSemanticsofImplication}.
\end{proof}

\begin{cor}[Persistence] \label{lemma:ElementaryPersistenceLemma}
    If $w\Vdash x \Rightarrow p$ holds and $p\Rightarrow q = \top$, then $w\Vdash x\Rightarrow q$ holds too.
\end{cor}
\begin{proof}
    Use the previous corollary and transitivity of implication.
\end{proof}

Now, the preliminaries of \S \ref{section:HeytingValuedSafety} can be reproved in the present context.

\begin{lemma}[Forward Consistency] 
    For any execution $f\colon w\to v$, if $p$ is safe in $w$, then $p$ is safe in $v$. That is, in notation, if $w\Vdash (x\Rightarrow p)$ then $v\Vdash (x\Rightarrow p)$.
\end{lemma}
\begin{proof} 
    Any execution $f\colon w\to v$ induces a morphism $f^*\colon \Sigma(v,-)\to\Sigma(w,-)$ and conversely. This is basically the internalized fibered Yoneda lemma (cf. \S B2.7 \cite{ElephantI} for example) although it can be worked out by hand using universal properties of pullbacks. By forcing stability \ref{lemma:ForcingStability}, $v\Vdash (x\Rightarrow p)$ then holds.
\end{proof}

\begin{lemma}[Current Consistency] 
    If $p$ is safe in $w$, then $\neg p$ is not safe in $w$. That is, if $w\Vdash (x\Rightarrow p)$ holds then $\sim (w\Vdash  (x \Rightarrow \neg p))$.
\end{lemma}
\begin{proof} 
    If $w\Vdash (x\Rightarrow \neg p)$ holds, then in particular $ew \leq \neg p$ does too. Consequently, by contraposition, if $ew\nleq \neg p$ holds, then $\sim (w\Vdash (x\Rightarrow \neg p))$. In general if $ew \leq p$, then $ew\nleq \neg p$ by assumption on $e$. Therefore, since $w\Vdash (x\Rightarrow p)$ implies in particular that $ew\leq p$, the conclusion follows.  
\end{proof}

\begin{lemma}[Backward Consistency] 
    For all executions $f:w\to v$, if $p$ is safe in $v$, then $\neg p$ is not safe in $w$. That is, if $v \Vdash  ( x\Rightarrow p)$ holds, then $\sim w \Vdash (x\Rightarrow\neg p)$.
\end{lemma}
\begin{proof} 
    By the previous two lemmas, $w \Vdash (x\Rightarrow \neg p)$ implies that $\sim v\Vdash (x\Rightarrow \neg\neg p)$. Therefore, by contraposition, $\sim\sim(v\Vdash (x\Rightarrow \neg\neg p))$ implies that $\sim (w \Vdash (x\Rightarrow \neg p))$. But note that $v \Vdash (x\Rightarrow p)$ implies that $v \Vdash (x\Rightarrow \neg\neg p)$ by Corollary \ref{lemma:ElementaryPersistenceLemma} since $p\leq \neg\neg p$ always holds. But this implies that $\sim\sim(v \Vdash (x\Rightarrow \neg\neg p))$ holds too. Putting together these implications, the result then follows. 
\end{proof}

The main result of the paper is now the following.

\begin{theo}[Estimate Safety] \label{theoToposSafety}
Inconsistent propositions are not safe at related states. That is, if $p\wedge q = \bot$ and $w_1\simeq w_2$ both hold, then it is not the case that both $w_1\Vdash (x\Rightarrow p)$ and $w_2\Vdash (x\Rightarrow q)$ hold. 
\end{theo}
\begin{proof} 
    By Lemma \ref{HeytingLemma2}, $p\wedge q = \bot$ is equivalent to $q\leq \neg p$ which is equivalent to $q\Rightarrow \neg p=\top$. Now, by Forward Consistency and Backward Consistency, the implications
        \begin{equation} \label{equation1'}
            w_1\Vdash (x\Rightarrow p) \supset w_3\Vdash (x\Rightarrow p)
                                   \supset \sim(w_2\Vdash (x\Rightarrow\neg p))
        \end{equation} 
    hold. By the Persistence Lemma \ref{PersistenceLemma2} applied to $q\Rightarrow \neg p$ and contraposition,
        \begin{equation} \label{equation2'}
            (\sim w_2\Vdash(x\Rightarrow \neg p)) \supset (\sim w_2\Vdash (x\Rightarrow q)).
        \end{equation} 
    Therefore, putting together the implications in \ref{equation1'} and \ref{equation2'},
        \[ 
            (w_1\Vdash (x\Rightarrow p)) \supset (\sim w_2\Vdash (x\Rightarrow q))
        \]
holds metatheoretically. Thus, if $p$ is safe in $w_1$ then $q$ cannot be safe in $w_2$. The argument is perfectly symmetric, hence the roles of $p$ and $q$ and of $w_1$ and $w_2$ can be interchanged. Therefore, not both $w_1\Vdash (x\Rightarrow p)$ and $w_2\Vdash (x\Rightarrow q)$ as required.
\end{proof}

\section{Decided Propositions in a Topos}
\label{section:DecidedPropositionsInaTopos}

Decided properties of protocol states are considered for example in \cite{ZamAndFam}. Once safety is established, protocol states can be decided. Here the notion is formalized in a topos $\basetopos$.

\begin{define} \label{DefnDecided}
    A property $p\colon \Sigma \to \Omega$ is \textbf{decided} for a state $w\colon W\to \Sigma$ if it is valid in all future states accessible from $w$. That is, $p$ is decided for $w$ if for all $w\to v$, it follows that $pv = \top$.
\end{define}

In this set up, a forcing interpretation of decided propositions is a special case of that for estimate safety in Theorem \ref{theorem:ElementaryForcingSemanticsOfSafety}. Recall that `$w$' on the left of `$\Vdash$' is a shorhand for $\Sigma(w,-)$.

\begin{cor}
    A proposition $p$ is decided in $w$ if, and only if, $w\Vdash p$ holds.
\end{cor}
\begin{proof}
    The proposition $p$ plays the role of $e$ in Theorem \ref{theorem:ElementaryForcingSemanticsOfSafety} since $P1 = \Omega$ and $p = p\Rightarrow\top$. In other words, the equivalences
        \begin{align}
            w\Vdash p &\equiv w\Vdash p \Rightarrow \top \notag \qquad &\text{($p=p\Rightarrow\top$)} \\
                        &\equiv v\Vdash p \Rightarrow \top \text{ for all } w\to v \notag \qquad &\text{(Theorem \ref{theorem:ElementaryForcingSemanticsOfSafety})} \\
                        &\equiv v\Vdash p \text{ for all } w\to v \notag \qquad &\text{($p=p\Rightarrow\top$)} \\
                        &\equiv pv =\top \text{ for all } w\to v \notag \qquad &\text{(def `$\Vdash$')} 
        \end{align}
    prove the result.
\end{proof}

As a result, write `$w\Vdash p$' to indicate that $p$ is decided in $w$. Theorem \ref{theoToposSafety} implies that inconsistent propositions cannot be decided in states with a common future. 

\begin{cor}
    Inconsistent state-propositions cannot both be decided in related protocol states. That is, if two states $w_1$ and $w_2$ have a common future state and if $p\wedge q =\bot$ holds, then $p$ and $q$ cannot both be decided in states $w_1$ and $w_2$, respectively, that is, $w_1\Vdash p$ and $w_2\Vdash q$ cannot both hold.
\end{cor}
\begin{proof}
    Since $w\Vdash p$ is equivalent to $w\Vdash p\Rightarrow \top$, Theorem \ref{theoToposSafety} proves the result.
\end{proof}

Close with a modal interpretation of decided states. This is more natural since $\Sigma\to 1$ should always be an epimorphism. First work over $\set$. Let $\F = [\Sigma,\set]$ and $\Gamma\colon \F \to\set$ denote the canonical global sections geometric morphism with inverse image given by the diagonal presheaf functor (cf. \S I.6 \cite{MM}). There is then an adjoint modality
    $$\xymatrix{
        \Gamma(1,\Omega_\F) \ar[r]^-\tau & \mathbf 2 \ar[r]^-i & \Gamma(1,\Omega_\F)
    }$$
yielding the modal operator $\Box = i\tau$. Note that $i$ is monic. View the terminal object in $\F$ as given by $w\mapsto \lbrace w\rbrace$. Suppose that the proposition $p\colon \Sigma_0\to \mathbf 2$ extends to a functor $\Sigma\to\mathbf 2$. In this case p determines one $p\colon 1\to\Omega_\F$ in $\F$ by the assignments
    \[
        w \mapsto \begin{cases}
            \mfrk t_w\qquad &\text{if } pw=1 \\
            \emptyset \qquad &\text{otherwise}.
        \end{cases}        
    \] 
Interpret $\Box p\colon 1\to \Omega_\F$ as its image under $\Box$. There is then the following result.

\begin{prop}
    A proposition $p\colon 1\to \Omega$ is decided in state $w$ if and only if $w\Vdash_* \Box p$ holds.
\end{prop}
\begin{proof}
    The computation
    \begin{align}
        w\Vdash_* \Box p &\equiv v\Vdash_* p \text{ for all } w\to v \notag \qquad & \text{(Theorem \ref{theorem:BoxSemanticsElementaryVersion})} \\
                       &\equiv v\Vdash \bar p \text{ for all } w\to v \notag \qquad & \text{(def. `$\Vdash_*$')} \\
                       &\equiv \epsilon_v p = \mfrk t_v \text{ for all } w\to v \notag \qquad & \text{(def. `$\Vdash$')} \\
                       & \equiv p(v) = \mfrk t_v \text{ for all } w\to v & \text{(constr. $\epsilon$)}\notag
    \end{align}    
    proves the result by the construction of the transpose.
\end{proof}

\section{Prospectus}

This study is not meant explicitly to advance the practical implementation of consensus protocols in the various languages in which they are written. The hope has at least been to illustrate applicability of topos theory to the description of those protocols and introduce these ideas especially to the applied category theory community. At most there is the possibility that such descriptions clarify the issues in such a way as to facilitate future developments in consensus protocol design and implementation.  

Mathematically there is work to be done mostly on the questions raised implicitly in Remark \ref{remark:SheafStuff}. This is that of whether sheaves play a role in the forcing semantics developed in that subsection. This seems likely, although somewhat awkward to work out. This is because the cosieves arising in the subobject classifier for copresheaves has played a central role. Passing to presheaves and then onto sheaves introduces contravariance, hence ordinary sieves, which eliminates the item playing the most important role in the present semantics. However, thinking of copresheaves and presheaves as forming the algebraic and geometric sides of some abstract ``Isbell Duality,'' it seems plausible that there is a way of reworking the results of this paper either to make sense for sheaves, or perhaps ``to cosheafify'' the present covariant version without passing to the explicitly geometric side of the duality.

\bibliographystyle{alpha}

\end{document}